\documentclass[11pt]{article}
\usepackage[utf8]{inputenc}
\usepackage{amssymb, amsmath, amsthm, url}
\usepackage{hyperref}
\usepackage{verbatim, color}

\newtheorem{theorem}{Theorem}[section]
\newtheorem{lemma}[theorem]{Lemma}
\newtheorem{conjecture}[theorem]{Conjecture}
\newtheorem{problem}[theorem]{Open Problem}
\newtheorem{proposition}[theorem]{Proposition}
\newtheorem{corollary}[theorem]{Corollary}

\theoremstyle{definition}
\newtheorem{definition}[theorem]{Definition}

\theoremstyle{remark}
\newtheorem{remark}[theorem]{Remark}
\newtheorem{question}[theorem]{Question}

\newcommand{\bbP}{\ensuremath{\mathbb{P}}}
\newcommand{\bbF}{\ensuremath{\mathbb{F}}}

\newcommand{\bbC}{\ensuremath{\mathbb{C}}}

\newcommand{\cO}{\ensuremath{\mathcal{O}}}

\newcommand{\Aut}{\ensuremath{\operatorname {Aut}}}
\newcommand{\Sing}{\ensuremath{\operatorname {Sing}}}

\newcommand{\mult}{\ensuremath{\operatorname {mult}}}

\newcommand{\SL}{\ensuremath{\operatorname {SL}}}
\newcommand{\PSL}{\ensuremath{\operatorname {PSL}}}
\newcommand{\PGL}{\ensuremath{\operatorname {PGL}}}

\newcommand{\Orb}{\ensuremath{\mathbf{O}}}

\title{The 21 reducible polars of Klein's quartic}
\author{Piotr Pokora and Joaquim Roé}
\date{}

\begin{document}
\maketitle

\begin{abstract}
We describe the singularities and related properties of the arrangement of 21 reducible polars of Klein's quartic, containing Klein's well-known arrangement of $21$ lines. 
\end{abstract}

\section*{Introduction}
In 1878/1879, F. Klein \cite{Kle79} found and studied in detail 
a remarkable complex algebraic curve,
which nowadays bears his name. It can be defined 
in the complex projective plane
by the homogeneous polynomial
$$\Phi_{4}: x^{3}y + y^{3}z +z^{3}x,$$
and its exceptional properties follow from the fact that it
has an automorphism group of the maximal possible order 
according to Hurzwitz's bound \cite{Hur}, namely
$$|\Aut(\Phi_{4})| = 84(g-1)= 168.$$
Because $\Phi_{4}$ is a smooth quartic, its group of automorphisms $\Aut(\Phi_{4})$ is realized by a subgroup $G\subset\Aut(\bbP^2_\bbC)=\operatorname{PGL}(3,\bbC)$. 

In the theory of complex arrangements of lines, there is a particularly symmetric arrangement of 21 lines well known for its properties, extremal in many senses; it was discovered by Klein, and it is the unique set of 21 lines invariant under the action of $G$. 
In this work, following the recent trend to extend the study of arrangements of lines to arrangements of curves of higher degree, we study an arrangement of lines and conics closely related to Klein's arrangement of lines, which displays a similarly rich geometry, with some features extreme among arrangements of lines and conics.

Every sufficiently general smooth quartic has 21 reducible polars
(a fact already known by E.~Bertini in 1896, \cite{Bertini}) and in the case of Klein's quartic, each of them splits as the union of one of the lines in Klein's arrangement plus a smooth conic. 
Taken together, they form the arrangement $\mathcal{K}$ of 21 lines and 21 conics which we are going to describe.
Our main contribution is the description of all singularities of the arrangement $\mathcal{K}$, namely their position with respect to Klein's quartic and the action of the group $G$. The arrangement has $42$ nodes, $252$ ordinary triple points and $189$ ordinary quadruple points.
We additionally show the following.
\begin{enumerate}
	\item $\mathcal{K}$ is neither free nor nearly free \cite{Dimca}.
	\item The logarithmic Chern slope of the arrangement of 21 conics is $E\approx 2.25$, which is the largest known value in the class of arrangements of curves of degree 2.
	\item The Harbourne index of $\mathcal{K}$ is $h(\mathcal{K})\approx -3.087$, the second most negative known value for reduced curves with ordinary singularities after Wiman's arrangement of lines, which has $h(\mathcal{W})\approx-3.358$.
	\item The ideal $I_2$ (respectively, $I_3$) of the set of all singular points (respectively, of points of multiplicity at least 3) in $\mathcal{K}$ provides examples of \emph{failure of containment} \cite{SzSz}, in the sense that $\displaystyle{I_m^{(3)}\not\subseteq I_m^{2}}$ with $m \in\{2,3\}$.
\end{enumerate}

The equations of the involved curves are all explicit, and all claims we make can be verified by a Computer Algebra System such as Singular \cite{DGPS}. In the Appendix we include scripts to do so. In the text, we give conceptual proofs for all claims except the failure of the containment for $I_2$ due to very involving computations.

Additionally, we study further configurations of curves of higher degree invariant under the action of the Klein group $G$, and also a remarkably analogous configuration of reducible polars of the plane sextic curve invariant under $A_6$, which contains Wiman's arrangement of lines $\mathcal{W}$ mentioned above. For the latter arrangement, after proving its existence we compute its singularities (some of which are not ordinary) and Harbourne index, which turns out to be $-3.38$.

The paper is structured as follows. In Section \ref{sec:klein-lines}, we recall classical results on Klein's quartic and properties of Klein's arrangement of lines from a viewpoint of the Klein group $G$. In Section \ref{sec:construction}, we describe the singularities of Klein's arrangement $\mathcal{K}$ of $21$ conics and $21$ lines. As we shall see, all singularities of the arrangement are ordinary. Starting from Section 3, we focus on possible applications of our arrangement. We compute the Harbourne index of the arrangement $\mathcal{K}$ (Section \ref{sec:h-index}) and also for its further iterations (Section \ref{sec:iteration}). In Section \ref{sec:freeness}, we show that Klein's conic-line arrangement is neither free nor nearly-free, in Section \ref{sec:chern} we put our configuration into the context of log-surfaces and we compute its Chern numbers. In Sections \ref{sec:containment} \& \ref{sec:iteration}, we focus on the containment problem for symbolic and ordinary powers of ideals, showing that certain sets of singular points of $\mathcal{K}$ and other invariant arrangements of higher degree provide new counterexamples to the containment $I^{(3)} \subset I^{2}$. Finally, in Section \ref{sec:wiman}, we study the analogous arrangement of reducible quintics invariant under $A_6$.

We work over the complex numbers.

\section{Klein's arrangement of lines}\label{sec:klein-lines}

Klein's quartic has been extensively studied (the book 
\cite{KLBook} contains a plethora of information and references 
on this beautiful subject, including a translation of Klein's
original paper) both for its intrinsic properties as a Riemann 
surface and as an algebraic plane curve, which is the point of view
we are interested in. 

Among the $168$ automorphisms of Klein's quartic, 
one can find exactly $21$ involutions, i.e.,
$\alpha \in \Aut(\Phi_{4})$ such that $\alpha \neq 1$ and 
$\alpha^{2} = 1$. Each of these $21$ involutions,
seen as an automorphism of $\mathbb{P}^{2}$,
is a \emph{harmonic homology} (see \cite[5.7.23]{Cas14}), with a line of fixed points
(its \emph{axis}) and an isolated fixed point (its \emph{center}). 
The resulting arrangement of $21$ lines (the axes) 
has exactly $21$ quadruple points (at the centers)
and $28$ triple points (which by duality correspond to the 28 bitangents 
to the quartic). 

We next recall the construction and properties of Klein's
arrangement of lines, following N.~Elkies \cite{Elkies} 
(see also \cite{SymmetricBU} by T.~Bauer, S.~Di~Rocco, B.~Harbourne, J.~Huizenga, A.~Seceleanu, and T.~Szemberg).

Let $G= \PSL(2,7)\cong\Aut(\Phi_{4})$ be the unique simple 
group of order $168$. Denoting $\zeta$ a primitive 7-th root of 1,
$G$ has an irreducible $3$-dimensional representation $\rho$ 
over $\mathbb{Q}(\zeta)$, given by three generators $g,h,i$
as follows:
$$\rho(g) = \begin{pmatrix}
\zeta^4  & 0 & 0\\
0 & \zeta^2 & 0 \\
0 & 0 & \zeta
\end{pmatrix}, \quad
\rho(h) = \begin{pmatrix}
0 & 1 & 0\\
0 & 0 & 1 \\
1 & 0 & 0
\end{pmatrix}$$
and
$$\rho(i)=\frac{2\zeta^4+2\zeta^2+2\zeta+1}{7}\begin{pmatrix}
\zeta-\zeta^6 & \zeta^2-\zeta^5 & \zeta^4-\zeta^3 \\
\zeta^2-\zeta^5 & \zeta^4-\zeta^3 & \zeta-\zeta^6 \\
\zeta^4-\zeta^3 & \zeta-\zeta^6 & \zeta^2-\zeta^5
\end{pmatrix}.$$
Note that all three matrices have determinant 1 and the element $i$ has order $2$.
This representation gives an embedding of $G$ into $\SL_3(\mathbb{Q}(\zeta))$.  By projectivizing, $G$ acts on $\mathbb{P}^2$ and on its dual,
$\check{\mathbb{P}}^2$. Note that a matrix $M\in \rho(G)\subset \SL_3(\mathbb{Q}(\zeta))$ takes the point $p=[a:b:c]$ to the point $[a':b':c']$ determined by
$(a',b',c')=(M(a,b,c)^T)^T$, but if $[A:B:C]\in\check{\mathbb{P}}^2$ is the point dual to the line $L=V(Ax+By+Cz)$, then $M$ takes $L$ to $L'=V(A'x+B'y+C'z)$, where
$(A',B',C')=((M^{-1})^T(A,B,C)^T)^T$; i.e., $(A',B',C')=(A,B,C)M^{-1}$.

\begin{remark}\label{rmk:transpose}
	Observe that the transposed matrix of each generator 
	$\rho(g)$, $\rho(h)$, $\rho(i)$ belongs to the group
	they generate, $\rho(G)$ which, abusing the notation, we 
	call simply $G$. 
	Therefore the projectivity 
	$\iota:\mathbb{P}^2\rightarrow\check{\mathbb{P}}^2$
	given in coordinates $[a:b:c]$ and $[A:B:C]$ by the 
	identity matrix, even though it is not equivariant,
	maps orbits to orbits. More precisely, points
	$[a:b:c]$ and $[a':b':c']$ lie in the same orbit
	under the action of $G$
	if and only if the lines $V(ax+by+cz)$ and
	$V(a'x+b'y+c'z)$ lie in the same orbit under the dual action.
\end{remark}
Consider the action of $G$ on the homogeneous coordinate ring
$S = \mathbb{C}[x,y,z]$ of $\mathbb{P}^2$. 
The ring $S^G$ of polynomials invariant under the action of $G$
is well-known since Klein's work \cite{Kle79}. 
It is generated by four polynomials
$\Phi_{4},\Phi_{6},\Phi_{14}$, and $\Phi_{21}$, where $\Phi_d$ has degree $d$.  The invariant $\Phi_{21}$ defines the line arrangement $\mathcal{K}_1=V(\Phi_{21})$.  The polynomials $\Phi_{4},\Phi_{6},\Phi_{14}$ are algebraically independent, and $\Phi_{21}^2$ belongs to the ring generated by $\Phi_{4},\Phi_{6},\Phi_{14}$.

The geometric significance of the invariants $\Phi_d$ is explained in
\cite{Elkies}. The polynomials $\Phi_{4}$ and $\Phi_{6}$ are uniquely
determined up to a constant factor; 
we take $\Phi_{4}$ to be Klein's quartic above, and $\Phi_{6}$ to be
$$\Phi_{6} = -\frac{1}{54} H(\Phi_{4}) = xy^5 + yz^5+zx^5-5x^2y^2z^2,$$ 
where $H(\Phi_{4})$ is the Hessian determinant
$$H(\Phi_{4}) := \begin{vmatrix} \partial^2 \Phi_{4}/\partial x^2&\partial^2 \Phi_{4}/\partial x \partial y&\partial^2 \Phi_{4}/\partial x \partial z\\
\partial^2 \Phi_{4}/\partial y \partial x&\partial^2 \Phi_{4}/\partial y^2&\partial^2 \Phi_{4}/\partial y \partial z\\
\partial^2 \Phi_{4}/\partial z\partial x&\partial^2 \Phi_{4}/\partial z \partial y&\partial^2 \Phi_{4}/\partial z^2\end{vmatrix}.$$
The degree $14$ part of $S^G$ is spanned by $\Phi_{14}$
and $\Phi_{4}^2\Phi_{6}$, so the invariant $\Phi_{14}$ is 
not uniquely determined up to constants.
One possible choice is 
$$\Phi_{14} = \frac{1}{9}BH(\Phi_{4},\Phi_{6}),$$ where $BH(\Phi_{4},\Phi_{6})$ is the \emph{bordered Hessian} $$BH(\Phi_{4},\Phi_{6}) := \begin{vmatrix} \partial^2 \Phi_{4}/\partial x^2&\partial^2 \Phi_{4}/\partial x \partial y&\partial^2 \Phi_{4}/\partial x \partial z & \partial \Phi_{6} / \partial x\\
\partial^2 \Phi_{4}/\partial y \partial x&\partial^2 \Phi_{4}/\partial y^2&\partial^2 \Phi_{4}/\partial y \partial z & \partial \Phi_{6}/\partial y\\
\partial^2 \Phi_{4}/\partial z\partial x&\partial^2 \Phi_{4}/\partial z \partial y&\partial^2 \Phi_{4}/\partial z^2 & \partial \Phi_{6}/\partial z\\
\partial \Phi_{6}/\partial x & \partial \Phi_{6}/\partial y & \partial \Phi_{6}/\partial z & 0 \end{vmatrix}.$$
Finally, the polynomial $\Phi_{21}$, which splits 
as the product of the lines in Klein's arrangement, can 
be defined by the following Jacobian determinant
$$\Phi_{21} =\frac{1}{14} J(\Phi_{4},\Phi_{6},\Phi_{14}) = \frac{1}{14} \begin{vmatrix} \partial \Phi_{4}/\partial x & \partial \Phi_{4}/\partial y & \partial \Phi_{4}/\partial z\\
\partial \Phi_{6}/\partial x & \partial \Phi_{6}/\partial y & \partial \Phi_{6}/\partial z\\
\partial \Phi_{14}/\partial x & \partial \Phi_{14}/\partial y & \partial \Phi_{14}/\partial z\\ \end{vmatrix}.$$

\begin{proposition}[\cite{Kle79}, see also \cite{Elkies},\cite{SymmetricBU}]\label{pro:orbits} 
	Every point $p\in \bbP^2$ has an orbit
	of size 168, except:
	
\begin{enumerate}
	\item The triple points of $\mathcal{K}_1$ form an orbit 
	$\Orb_{28}$ of size $28$. The point $[1:1:1]$ lies in this orbit.
	
	\item The quadruple points of $\mathcal{K}_1$ form an orbit $\Orb_{21}$ 
	of size $21$. 
	
	\item The invariant curves $V(\Phi_{4})$ and $V(\Phi_{6})$ meet in 
	an orbit $\Orb_{24}$ of $24$ points. 
	The point $[1:0:0]\in \bbP^2$ lies in this orbit.
	
	\item The invariant curves $V(\Phi_{4})$ and $V(\Phi_{14})$ meet 
	in an orbit $\Orb_{56}$ of $56$ points.  
	The point $[\omega^2:\omega:1]\in \bbP^2$ lies in this orbit, where $\omega$ is a primitive cube root of 1.
	
	\item The invariant curves $V(\Phi_{6})$ and $V(\Phi_{14})$ are tangent at an orbit $\Orb_{42}$ of $42$ points belonging to $\mathcal{K}_1$.
	
	\item Any point on $\mathcal{K}_1$ not mentioned above has an orbit of size $84$.
\end{enumerate}
Moreover, the orbits $\Orb_{21}$ and $\Orb_{28}$ are not contained in either
$V(\Phi_{4})$, $V(\Phi_{6})$ or $V(\Phi_{14})$. The orbit $\Orb_{24}$ 
is not contained in either $V(\Phi_{14})$ or $\mathcal{K}_1=V(\Phi_{21})$. 
The orbit $\Orb_{56}$ is not contained in either
$V(\Phi_{6})$ or $\mathcal{K}_1=V(\Phi_{21})$. 
\end{proposition}

By Remark \ref{rmk:transpose}, the orbits of lines follow
the same pattern as the orbits of points. 
In particular, there is a unique orbit of 21 lines,
which are the components of $\Phi_{21}$ (and by the resulting duality,
each of the 21 lines contains four points of the orbit of size 21).
Recall that for each line in 
$\Phi_{21}$ there is an involution $\alpha\in G$ which fixes it
pointwise --- for a general point $p$ in the line, $\alpha$
generates the stabilizer of $p$, and so $p$ belongs to one of the
orbits of size 84 described in the last item. The collection of centers 
of the 21 involutions is also an orbit, namely $\Orb_{21}$.

Similarly, there is an orbit of 28 lines, made up exactly of the 28
bitangents of $\Phi_{4}$ (every smooth quartic has 28 bitangents). 
The orbit $\Orb_{56}$ consists of the tangency points of $\Phi_4$ with its bitangents. 
The joint arrangement of 21+28 lines and points is in itself
a beautiful object represented combinatorially by the so-called
\emph{Coxeter graph}; the following description is taken from
\cite[Theorem 6.1]{Bit}. 

\begin{theorem}\label{thm:49}
Let $P$ be the set of points and $L$ the set of lines 
of the Fano plane $\bbP^2_{\bbF_2}$. 
There is a labelling of the $49$ singular points of $\Phi_{21}$, i.e.,
in $\Orb_{21} \cup \Orb_{28}$, by the elements of $P \times L$ and a
labelling of the 21 lines in $\Phi_{21}$ and the $28$ bitangents 
by the elements of $L \times P$ in such a way that
\begin{enumerate}
\item $\langle p, l \rangle  \in P \times L$ represents a point in 
$\Orb_{21}$ iff $p \in l, \langle l, p \rangle \in L \times P$ 
represents a component of $\Phi_{21}$ iff $l \ni p$,
\item the involution which fixes the line $\langle p, l \rangle$
has center $\langle l, p \rangle$,
\item the point $\langle p, l \rangle $ is on the bitangent $\langle m, q \rangle$ iff $p \in m$ and  $q \in l$,
\item the point $\langle p, l \rangle\in \Orb_{21}$ is on the component
$\langle m, q \rangle$ of $\Phi_{21}$ 
iff $p=q$ and $l\neq m$, or $p\neq q$ and $l=m$,
\item the point $\langle p, l \rangle\in \Orb_{28}$ is on the component
$\langle m, q \rangle$ of $\Phi_{21}$ 
iff $p \in m$ and $q \in l$.
\end{enumerate}
There are no other incidences between the $49$ points and the $49$ lines.
\end{theorem}

\section{The arrangement of 21 conics and 21 lines}\label{sec:construction}
It is known that a general smooth quartic has 21 reducible polar curves, namely the polars with respect to the 21 nodes of its Steinerian curve.
The Steinerian curve of a given curve $C$, introduced by J.~Steiner in \cite{Steiner}, is the locus of points $p$ such that the polar of $C$ with respect to $p$ is singular. A modern presentation of the Steinerian and its role in the polarity theory was given by I.~Dolgachev in \cite[1.1]{Dolgachev}.

It is well known (at least since 1896 by Bertini \cite{Bertini}) that the Steinerian of a general quartic has 21 nodes and 24 cusps; this is also the case for Klein's quartic (see \cite[Example~6.1.1]{DK93}, \cite[\S15]{Adler} for modern proofs and references to the history of the subject). Additionally, the Steinerian is invariant under the action of $G$, so its 21 nodes are invariant, hence they are the points of $\Orb_{21}$. 
Note also that, since the Steinerian is an invariant curve of degree 12, its equation must be a combination of $\Phi_{4}^3$ and $\Phi_{6}^2$. 
The only linear combination of these forms which vanishes at $\Orb_{21}$ is $4\Phi_{4}^3+\Phi_{6}^2$, so this is the equation of the Steinerian.

The polar of $\Phi_{4}$ with respect to each point $p\in \Orb_{21}$ splits as a smooth conic and a line meeting transversely at two points. 
To describe in detail the singularities of the arrangement given by the 21 reducible polars, we shall use the gradient map $\nabla(\Phi_{4})$ given 
by the partial derivatives of Klein's quartic equation, namely
$$\mathbb{P}^{2} \ni [x:y:z] \overset{9:1}{\longmapsto} [u:v:w]=\left[3x^{2}y + z^{3}: 3y^{2}z + x^{3}: 3z^{2}x+y^{3} \right] \in \check{\mathbb{P}}^{2}.$$
By definition, the polar $P_p(\Phi_4)$ with respect to the point $p=[a:b:c]\in \bbP^2$ is the preimage under $\nabla(\Phi_{4})$ of the line
dual to $p$, $V(au+bv+cw)\subset \check{\mathbb{P}}^{2}$. 
Since Klein's quartic is non-singular, the partial derivatives
never vanish simultaneously, so the gradient map
is defined everywhere, i.e., it is a morphism.
Equivalently, the net of polar curves to $\Phi_{4}$,
which are the pullbacks of the lines 
by $\nabla(\Phi_{4})$, has no base points.

The morphism $\nabla(\Phi_{4})$ is $9:1$, and it is  
ramified along the degree $6$ smooth Hessian curve $V(\Phi_{6})$. 
The image of Klein's quartic is its dual, and the image of the Hessian is the discriminant, or branch curve, of the morphism, which we denote $\Delta=\nabla(\Phi_{4})(V(\Phi_{6}))$.
By definition, the Steinerian is the dual of the discriminant $\Delta$ (see also \cite[1.1]{Dolgachev}).

It is well-known that the gradient map
is covariant under the representation $\rho$, respectively its dual ---
in fact this follows from the observation
that the linear space spanned by the partials 
is invariant under the action of $\rho$ on $S$. 
Hence $\nabla(\Phi_{4})$ maps orbits for the representation $\rho$ to orbits for the dual representation.
For convenience we identify $\bbP^2$ with its dual via the isomorphism $\iota:\bbP^2\rightarrow \check{\mathbb{P}}^{2}$ defined in Remark \ref{rmk:transpose} and given in coordinates by the identity matrix; recall that with this identification, 
the orbits by $\rho$ and its dual agree, and we shall simply say that $\nabla(\Phi_{4})$ maps orbits to orbits.

\begin{proposition}\label{pro:reducible_polars}
	$\Phi_{4}$ has exactly $21$ reducible polars, and each consists of a line
	and a conic meeting transversely at two points. 
	They are the polars with
	respect to the $21$ quadruple points of $\Phi_{21}$, and 
	the preimages by $\nabla(\Phi_{4})$ of the $21$ lines making up
	$\Phi_{21}$. The $21$ lines which are components of reducible
	polars are also the components of $\Phi_{21}$, whereas each of the 
	$21$ conics intersects the Klein curve at the eight points of contact
	of four bitangents. The nodes of the 21 reducible polars 
	are all distinct and form the orbit $\Orb_{42}$.
\end{proposition}

\begin{proof}
	The fact that the $21$ reducible polars in the net are the polars with respect to the 21 nodes of the Steinerian, and that they split as transverse conic-line pairs, was shown in \cite[Example~6.1.1]{DK93}. As we already mentioned, their invariance under $G$ implies that the 21 nodes of the Steinerian are
	the points in $\Orb_{21}$. 
	By definition, the polars of the $21$ points coincide with the
	pullbacks by $\nabla(\Phi_4)$ of the $21$ lines. Also, the $21$ line components of the
	reducible polars form an orbit under the action of $G$,
	and the only orbit of size $21$ is made up of the components
	of $\Phi_{21}$.
	
	If $p\in\Orb_{21}$, by Theorem \ref{thm:49}, four
	bitangents of $V(\Phi_{4})$ meet at $p$, so their eight points of
	tangency to $V(\Phi_{4})$ belong to the polar $(\Phi_{4})_p$, which we know splits as a pair conic and line. 
	By Theorem \ref{thm:49} and Proposition \ref{pro:orbits},
	the eight bitangency points, which belong to $\Orb_{56}$, do not lie on the $21$ lines of $\mathcal{K}_1$, so
	they belong to the conic 
	component.
	
	The whole set of singular points of the
	$21$ reducible polars (two points on each) is invariant under 
	$G$. So these singularities form either the orbit $\Orb_{42}$
	or the orbit $\Orb_{21}$ (in which case each point would
	be singular for two reducible polars). 
	But the latter option leads to a contradiction:
	a reducible polar is the preimage of a line
	by $\nabla(\Phi_{4})$, so every singularity on it must belong to 
	the ramification locus $\Phi_{6}$, whereas, by Proposition \ref{pro:orbits},
	$\Phi_{6}$ does not contain $\Orb_{21}$. So the singularities
	of the 21 reducible polars are all distinct, and they are exactly
	the 42 points in $\Orb_{42}$.
\end{proof}

\begin{definition}
	We shall denote by $\Phi_{63}=(\nabla(\Phi_{4}))^{*}(\Phi_{21})$ 
	the equation of the arrangement of 21 reducible polars.
	It splits as $\Phi_{63}=\Phi_{21}\Phi_{42}$, where $\Phi_{42}$
	defines the arrangement $\mathcal{K}_2=V(\Phi_{42})$ of 21 conics.
	We also denote $\mathcal{K}=\mathcal{K}_1+\mathcal{K}_2=V(\Phi_{63})$. 
\end{definition}

Our next goal is to determine the singularities of the reduced arrangement
$\mathcal{K}$. Clearly, these are the singularities of the reducible
polars (which were described in Proposition \ref{pro:reducible_polars}) 
and their intersection points, which
are the preimages of the singularities of $\mathcal{K}_1$.
These will be determined by showing that they lie off the 
ramification curve $V(\Phi_{6})$.

\begin{lemma}\label{lem:conincident}
	If $p$ is a point belonging to more than one reducible polar 
	of $\Phi_{4}$, then either $\nabla(\Phi_{4})(p)\in \Orb_{21}$
	and $p$ belongs to exactly $4$ reducible polars, or 
	$\nabla(\Phi_{4})(p)\in \Orb_{28}$
	and $p$ belongs to exactly $3$ reducible polars.
\end{lemma}
\begin{proof}
	The reducible polars are exactly the preimages by $\nabla(\Phi_{4})$
	of the lines composing $\mathcal{K}_1$, therefore a point as in
	the statement belongs to the preimage of two such lines.
	This implies that $p$ belongs to the preimage of their unique 
	intersection, which is either in $\Orb_{21}$ or $\Orb_{28}$.
	Moreover, the reducible polars to which $p$ belongs 
	are precisely the preimages of the line components
	going through $\nabla(\Phi_{4})(p)$.
\end{proof}

\begin{lemma}\label{lem:bitangency_points}
	The points in the orbit $\Orb_{56}$ are ordinary singularities
	of multiplicity 3 in $\mathcal{K}$ and in $\mathcal{K}_2$. 
	Moreover $\nabla(\Phi_{4})(\Orb_{56})=\Orb_{28}$.
\end{lemma}
\begin{proof}
	Since $\Phi_{21}$ does not vanish at $\Orb_{56}$ 
	(Proposition \ref{pro:orbits}), the singularities of 
	$\mathcal{K}$ and $\mathcal{K}_2$ at $\Orb_{56}$ are equal.

	$\Orb_{56}$ is the orbit formed by the tangency points
	of the 28 bitangents to Klein's quartic. As was shown in the proof of Proposition \ref{pro:reducible_polars}, each (conic) component of
	$\mathcal{K}_2$ goes through 8 points of $\Orb_{56}$, and there
	are 21 such components, so there are $3=21\times 8 / 56$
	components through each point. 
	
	To see that the triple points
	of $\mathcal{K}$ at each point of 
	$\Orb_{56}$ are ordinary, we use the 
	gradient map $\nabla(\Phi_{4})$. 	
	Since $\Phi_{6}$ does not vanish at $\Orb_{56}$, 
	$\nabla(\Phi_{4})$ is biholomorphic in a neigbourhood of each
	point of $\Orb_{56}$. Thus the triple points of $\Phi_{63}$
	are analytically isomorphic to the triple points of $\Phi_{21}$, which 
	--- being unions of lines --- are ordinary triple points,
	and they are located at $\Orb_{28}$.
\end{proof}

\begin{proposition}\label{preimage_28}
	Each point in $\Orb_{28}$ has $9$ distinct preimages by
	the morphism $\nabla(\Phi_{4})$. The whole preimage
	$\nabla(\Phi_{4})^{-1}(\Orb_{28})$ consists of
	$\Orb_{28}$, $\Orb_{56}$, and two orbits of size $84$.
\end{proposition}

\begin{proof}
	We will describe the preimage of $p=[1:1:1]\in \Orb_{28}$. 
	It consists of the base locus of the pencil of cubics obtained
	as preimages of lines through $p$. The equations of the
	lines composing $\Phi_{21}$ are well-known, and the three
	going through $p$ are given by the vanishing of the forms
	\begin{align*}
	\Lambda_1&= x+(\zeta^4+\zeta^3)y-(\zeta^4+\zeta^3+1)z,\\
	\Lambda_2&= -(\zeta^4+\zeta^3+1)x+y+(\zeta^4+\zeta^3)z,\\
	\Lambda_3&= (\zeta^4+\zeta^3)x-(\zeta^4+\zeta^3+1)y+z,
	\end{align*}
	where we note that $\Lambda_1+\Lambda_2+\Lambda_3=0$. Thus we
	are looking for the base points of the pencil generated by
	the preimages of the $\Lambda_i$. 

	The point $p$ is mapped to itself by $\nabla(\Phi_{4})$, so the linear
	components of the reducible polars $V(\nabla(\Phi_{4})^{-1}(\Lambda_i))$
	are exactly the lines $V(\Lambda_i)$. 
	Since $\Phi_{6}$ does not vanish at $p$, the singularity of
	$\mathcal{K}$ at $p$ is isomorphic to that of $\mathcal{K}_1$,
	i.e. none of the conics goes through $p$.
	
	The explicit computation
	of $\nabla(\Phi_{4})^*(\Lambda_i)$ shows that $\nabla(\Phi_{4})$
	maps each line $V(\Lambda_i)$ onto itself, and provides the equations 
	of the three conics mapping to the lines. The first of these is
	\begin{align*}
	\Gamma_1&=(\zeta^4+\zeta^3)x^2+(\zeta^5+\zeta^2)(\zeta^5+\zeta^2+1)
	xy+(\zeta^5+\zeta^2)y^2\\ & \quad +(-\zeta^5-\zeta^2+1)xz+
	(-\zeta^4-\zeta^3+1)yz+(\zeta^5+\zeta^2)(\zeta^4+\zeta^3-1)z^2,
	\end{align*}
	whereas the remaining two, $\Gamma_2$, $\Gamma_3$ are similar, with
	$x, y,z$ cyclically permuted. 
	
	By Lemma \ref{lem:conincident}, every point in the intersection
	of $V(\Lambda_2)$ with $V(\Gamma_1)$ is a triple point of
	$\mathcal{K}$ mapping to $p$, and since it is different from $p$
	it must belong to $V(\Gamma_3)$. More generally	
	
	\[ V(\Lambda_i)\cap V(\Gamma_j)=V(\Lambda_i) \cap V(\Gamma_k), \quad
	\{i,j,k\}=\{1,2,3\}, \]
	and each of these intersection points is a triple point of $\mathcal{K}$.
	If we prove that the intersection $V(\Lambda_2)\cap V(\Gamma_1)$
	consists of two distinct points, it will follow that
	$p$ has 9 distinct preimages as claimed, namely $p$ itself, two points of
	$\Orb_{56}$ (which by Lemma \ref{lem:bitangency_points}
	must be shared by the three conics $V(\Gamma_i)$) and two
	distinct points on each line $V(\Lambda_i)$.
	
	Plugging the value of $y$ from $\Lambda_2=0$ into $\Gamma_1=0$
	we can explicitly determine the two intersection points as
	the solution of a quadratic form in $x,z$ with nonvanishing
	discriminant.
	
	So the two intersection points are indeed distinct, which means
	that the conics $V(\Gamma_j)$, $V(\Gamma_k)$ intersect the line
	$V(\Lambda_i)$ transversely. Moreover, at these two points the conics
	$V(\Gamma_j)$, $V(\Gamma_k)$ cannot be tangent, because they 
	intersect in two additional points of $\Orb_{56}$. 
	Therefore, the three components $V(\Gamma_j)$, $V(\Gamma_k)$, 
	$V(\Lambda_i)$ meet at two ordinary triple points of $\mathcal{K}$.
	
	Each linear component $\Lambda$ of $\Phi_{21}$ 
	vanishes at 4 points of $\Orb_{28}$, and we just 
	showed that each of them has two preimages
	on $V(\Lambda)$ distinct from $p$ itself. This makes for a total of
	$4\cdot 2 \cdot 21=168$ points; by Proposition \ref{pro:orbits}
	these comprise two orbits of size $84$.
\end{proof}
\begin{proposition}
\label{pro:preimage21}
	Each point in $\Orb_{21}$ has $9$ distinct preimages
	by the morphism $\nabla(\Phi_{4})$. The whole preimage
	$\nabla(\Phi_{4})^{-1}(\Orb_{21})$ consists of $\Orb_{21}$
	and two orbits of size $84$.
\end{proposition}

\begin{proof}
    We will argue analogously to Proposition \ref{preimage_28},
    and describe the preimage of 
		\[p=[1:\zeta^4+\zeta^3:-(\zeta^4+\zeta^3+1)]\in \Orb_{21}.\] 
	The equations of the
	lines composing $\Phi_{21}$ are well-known, and the four
	going through $p$ are given by the vanishing of the forms
	\begin{align*}
	\Lambda_1&= x+(\zeta^5+\zeta)y+(\zeta^5+\zeta^4+\zeta^2+1)z, \\
	\Lambda_2&= x-(\zeta^5+\zeta^4+\zeta^3+\zeta+1)y+(\zeta^5+\zeta^3+\zeta^2+1)z,\\
	\Lambda_3&= x+(\zeta^5+1)y-(\zeta^3+\zeta^2+\zeta)z,\\
	\Lambda_4&= x+(\zeta^2+1)y+(\zeta^3+\zeta^2+\zeta+1)z.
	\end{align*}
	We are looking for the base points of the pencil generated by
	the preimages of the $\Lambda_i$. 
	
	As with Proposition \ref{preimage_28}, 
	none of the conics goes through $p$.
	The explicit computation
	of $\nabla(\Phi_{4})^*(\Lambda_i)$ shows that $\nabla(\Phi_{4})$
	exchanges the lines $V(\Lambda_i)$ in sets of two,
	$\Lambda_1\leftrightarrow\Lambda_2$, 
	$\Lambda_3\leftrightarrow\Lambda_4$, and provides the equations 
	of the four conics mapping to the lines. The first of these is given by
	\begin{gather*}
	\Gamma_1= (\zeta^5+\zeta)x^2-(\zeta^4+\zeta^3-1)xy+(\zeta^5+\zeta^3)y^2+
	\\(-\zeta^3+\zeta^2-\zeta)xz+(\zeta^6-\zeta^4-\zeta)yz+(\zeta^5+\zeta^4)z^2	.
	\end{gather*}
	By Lemma \ref{lem:conincident}, every point in the intersection
	of $V(\Lambda_2)$ with $V(\Gamma_1)$ is a quadruple point of
	$\mathcal{K}$ mapping to $p$, and since it is different from $p$,
	it must belong to the other two conics $V(\Gamma_3)$ and 
	$V(\Gamma_4)$. More generally	
	\[ V(\Lambda_i)\cap V(\Gamma_j)=V(\Lambda_i) \cap V(\Gamma_k)=
	V(\Lambda_i)\cap V(\Gamma_{\ell}), \quad
	\{i,j,k,\ell\}=\{1,2,3,4\}, \]
	and each of these intersection points is a quadruple
	 point of $\mathcal{K}$.
	If we prove that the intersection $V(\Lambda_2)\cap V(\Gamma_1)$
	consists of two distinct points, it will follow that
	$p$ has $9$ distinct preimages as claimed, namely $p$ itself and two
	distinct points on each line $V(\Lambda_i)$.
	
	Plugging the value of $y$ from $\Lambda_2=0$ into $\Gamma_1=0$,
	we can explicitly determine the two intersection points as
	the solution of a quadratic form in $x,z$ with non-vanishing
	discriminant.
	So the two intersection points are indeed distinct, which means
	that the conics $V(\Gamma_j)$, $V(\Gamma_k)$ intersect the line
	$V(\Lambda_i)$ transversely. Moreover, at these two points the conics
	$V(\Gamma_j)$, $V(\Gamma_k)$ cannot be tangent, because they 
	intersect in two additional points of $V(\Lambda_{\ell})$. 
	Therefore, the four components $V(\Gamma_j)$, $V(\Gamma_k)$,
	$V(\Gamma_{\ell})$ and $V(\Lambda_i)$ meet at two ordinary 
	quadruple points of $\mathcal{K}$.
	
	Each linear component $\Lambda$ of $\Phi_{21}$ 
	vanishes at 4 points of $\Orb_{21}$, and we just 
	showed that each of them has two preimages
	on $V(\Lambda)$ distinct from $p$ itself. This makes for a total of
	$4\cdot 2 \cdot 21=168$ points; by Proposition \ref{pro:orbits}
	these comprise two orbits of size $84$.
\end{proof}
 
\begin{corollary}
	The curve $\mathcal{K}$
	is reduced and has only transversal intersection 
	points as singularities, namely
	$189 = 9 \cdot 21$ quadruple points,
	$252 = 9 \cdot 28$ triple points, 
	and $42 = 2 \cdot 21$ nodes.
\end{corollary} 

\begin{corollary}
	The curve $\mathcal{K}_2$
	is reduced and has only transversal intersection 
	points as singularities, namely
	$224 = 8 \cdot 21 + 56$ triple points,
	and $168 = 8 \cdot 21$ nodes.
\end{corollary} 

\begin{remark}\label{rmk:cusps_discriminant}
	The discriminant $\Delta$ of the morphism $\nabla(\Phi_{4})$ has degree 18, as it is the image of a sextic under $\nabla(\Phi_{4})$. 
	Since $\Delta$ is the dual of the Steinerian, which is a curve of degree 12 with exactly 21 nodes and 24 cusps as singularities,
	Plücker's second formula gives the number of cusps of $\Delta$, namely 42.  
	A priori these 42 cusps could coalesce into a smaller number of more complicated singularities, but they must form a union of orbits; thus either $\Delta$ has exactly 42 ordinary cusps at the points of $\Orb_{42}$, or 21 singularities at $\Orb_{21}$. 
	On the other hand, by Proposition \ref{pro:preimage21}, we know that the discriminant does not pass through the points of $\Orb_{21}$, so $\Delta$ has ordinary cusps at the points of $\Orb_{42}$. 
	We shall check in \ref{rmk:cusps_K1} that the two cusps of $\Delta$ at the two singular points of each reducible polar have the same tangent line, namely the linear component through them. 
	This gives additional information on the Steinerian of the Klein curve: its 42 inflectional branches meet in pairs to form its 21 nodes.
\end{remark}

\section{Harbourne indices}\label{sec:h-index}

The arrangement of reducible polars of Klein's quartic turns out to be relevant in the context of the Bounded Negativity Conjecture and $H$-indices. 

\begin{conjecture}[BNC] \label{con:BNC}
	Let $X$ be a smooth complex projective surface. There exists a positive integer $b(X) \in \mathbb{Z}$ such that for all \emph{reduced} curves $C \subset X$ one has $C^{2} \geq -b(X)$.
\end{conjecture}

This conjecture is widely open, so much so that it is not even known whether the blowing-up of the complex projective plane along $r\geq 10$ points has bounded negativity.
However, in all \emph{known} cases of reduced curves in the blowing up of $\bbP^2$ at $r$ points, the self-intersection is bounded by $-4r$. 
This prompted the introduction of \emph{H-constants} \cite{BdRHHLPSz} and \emph{Harbourne indices}. 
\begin{definition}
	Let $C \subset \mathbb{P}^{2}_{\mathbb{C}}$ be a reduced curve of degree $d$, with ordinary singularities. The Harbourne index of $C$ is
	$$h(C) = \frac{d^{2} - \sum_{p\in \Sing(C)} \mult_{p}(C)^{2}}{|\Sing(C)|}.$$
\end{definition}
\begin{definition}
	The global linear Harbourne index of $\mathbb{P}^{2}$ is defined as
	$$H_{1}(\mathbb{P}^{2}) := \inf_{\mathcal{L}}h(\mathcal{L}),$$
	where the infimum is taken over all \emph{reduced} line arrangements $\mathcal{L} \subset \mathbb{P}^{2}$. 
\end{definition}

With these definitions in hand, the main result of \cite[Theorem 3.13]{BdRHHLPSz} can be formulated as follows. 
\begin{theorem}\label{thm:linear} One has
	$$H_{1}(\mathbb{P}^{2}) \ge -4.$$
\end{theorem}
It is natural to ask whether the above bound is sharp. The most negative known example is  \emph{Wiman's arrangement of lines} \cite{wiman}, the line arrangement of $45$ lines with $120$ triple points, $45$ quadruple points, and $36$ quintuple points. Easy computations reveal that for Wiman's arrangement $\mathcal{W}_1$ one has
$$h(\mathcal{W}_1) = -\frac{225}{67} \approx -3.358.$$
It is worth pointing out that Klein's arrangement of lines $\mathcal{K}_1$ delivers the second most negative known value of a linear Harbourne index, equal to $-3$.

Now we would like to focus on conic-line arrangements in the projective plane having only ordinary singularities. If $\mathcal{CL}$ is a conic-line arrangement consisting of $l$ lines and $k$ conics, the following combinatorial equality holds:
$$ {l \choose 2} + 2kl + 4 {k \choose 2} = \sum_{r\geq 2} {r \choose 2} t_{r},$$
where by $t_{r}$ we denote the number of $r$-fold intersection points, i.e., points where exactly $r$ curves from the arrangement meet, and we are using a convention that ${0 \choose 2} = {1 \choose 2} = 0$.
The above equality provides
$$(2k+l)^{2} -\sum_{r\geq 2}r^{2}t_{r} = 4k+l - \sum_{r\geq 2}rt_{r},$$
and shows that the Harbourne index of a conic-line arrangement $\mathcal{CL}$ with ordinary singularities satisfies
$$h(\mathcal{CL}) = \frac{4k+l - \sum_{r\geq2} rt_{r}}{| {\rm Sing}(\mathcal{CL})|}.$$

Not much is known about $H$-indices of conic-line arrangements. In a very recent paper \cite{Pokora2017}, the first author proved the following result.
\begin{theorem}(\cite[Theorem~2.1]{Pokora2017})\label{thm:pokora17}
	Let $\mathcal{CL} = \{L_{1}, ..., L_{l}, C_{1}, ..., C_{k}\}$ be an arrangement of $l$ lines and $k$ conics such that $t_{r} = 0$ for $r > \frac{2(l+2k)}{3}$. Then one has
	$$t_{2} + \frac{3}{4}t_{3} + (4k+2l-4)k \geq l + \sum_{r\geq 5}\bigg(\frac{r^{2}}{4}-r\bigg)t_{r}.$$
\end{theorem}
Theorem \ref{thm:pokora17} does not lead to any effective lower bound on Harbourne indices of conic-line arrangements. This leads to the following question.
\begin{question}
	How negative Harbourne indices for conic-line arrangements can be?
\end{question}
The Klein arrangement of conics and lines $\mathcal{K}$ has $21$ conics, $21$ lines, and $t_{2} = 42$, $t_{3} = 252$, and $t_{4} = 189.$
Therefore
$$h(\mathcal{K}) =-\frac{71}{23}\approx -3.0865,$$
which is the second smallest known value of Harbourne indices for arrangements of curves in the complex projective plane having transversal intersection points (here we exclude Cremona transform-arrangements constructed with use of line arrangements, for details please consult \cite{PTG}).
Note that the arrangement of conics $\mathcal{K}_2$, on its own, 
has Harbourne index equal to $-33/14 \approx -2.3571$.

\section{Freeness of arrangements}\label{sec:freeness}
A fundamental object associated to an arrangement of plane curves $\mathcal{F}$ is the module $D(\mathcal{F})\subset {\rm Der}_{\mathbb{C}}(S)$ of derivations tangent to the arrangement,
$$D(\mathcal{F}) = \{ \theta : \theta(\Phi) \in \langle \Phi \rangle \text{ for all } \Phi \text{ such that } V(\Phi) \in \mathcal{F}\}.$$
In his famous paper \cite{Terao}, Terao showed the link between Poincar\'{e} polynomials for line arrangements $\mathcal{L}$ and the freeness of the module $D(\mathcal{L})$, and it is notoriously difficult to predict whether the freeness of $D(\mathcal{L})$ is combinatorial in nature, which leads to Terao's conjecture. 

Next we focus on the freeness of conic-line arrangements, and in particular of the $G$-invariant arrangement $\mathcal{K}$, mostly in the spirit of \cite{SchT}. 
One might expect that $\mathcal{K}$ be free, given that 
Klein's arrangement of lines $\mathcal{K}_1$ is free as a reflection arrangement \cite[Chapter~6]{OT92}, $\mathcal{K}$ is its pull-back under the gradient map, and the whole construction is invariant under $G$. Note also that all singularities of $\mathcal{K}$ are quasihomogeneous. 
Indeed, nodes are always quasihomogeneous, whereas for all triple and quadruple points of $\mathcal{K}$ we showed in Section \ref{sec:construction} that they are biholomorphic to the triple and quadruple points of Klein's arrangement of lines, which are obviously homogeneous.
However, it will turn out that $\mathcal{K}$ \textbf{is not free}. 

Following \cite[Section 1.4]{SchT}, we know that if $\mathcal{CL}$ is a reduced conic-line arrangement, then
$$D(\mathcal{CL}) \simeq E \oplus D_{0}(\mathcal{CL}),$$
where $E$ is the Euler derivation and $D_{0}(\mathcal{CL})$ corresponds to the module of syzygies on the Jacobian ideal of the defining polynomial of $\mathcal{CL}$; the properties of $D_{0}(\mathcal{CL})$ determine freeness of the arrangement $\mathcal{CL}$.

We now prove that $\mathcal{K}$ is not free following an idea kindly suggested to us by the referee. Denote by $r$ the minimal degree such that the homogeneous component $D_{0}(\mathcal{K})_{r}$ is non-zero. 
By \cite[Collorary 1.2]{DimcaTjurina}, $\mathcal{K}$ is a free arrangement if and only if $\tau(\mathcal{K})= 62^{2} - r(62-r)$, where $\tau(\mathcal{K})$ is the total Tjurina number of $\mathcal{K}$. As we observed a few paragraphs earlier all singularities of $\mathcal{K}$ are quasihomogeneous, so the total Tjurina number is equal to the total Milnor number, i.e.,
$$\tau(\mathcal{K}) =  \mu(\mathcal{K}) = \sum_{p \in {\rm Sing}(\mathcal{K})} ({\rm mult}_{p} - 1)^{2} = 1 \cdot 42 + 4 \cdot 252 + 9 \cdot 189 = 2751.$$ 
This leads to the equation 
$$r^2 - 62r + 1093=0$$
which does not have real roots, so $\mathcal{K}$ cannot be free.

In a recent paper \cite{Dimca}, Dimca and Sticlaru introduced a larger class of arrangements, called \textit{nearly free curves}, which can be defined as follows.
\begin{definition}
\label{def:near}
	Let $\mathcal{F}$ be an arrangement of plane curves, then $\mathcal{F}$ is called \emph{nearly free} if the Jacobian module $N := I_{\mathcal{F}} / J_{\mathcal{F}}$
	does not vanish and its homogeneous pieces $N_k$ satisfy $\dim N_{k} \leq 1$ for any $k$. 
\end{definition} 
Following the same approach as for freeness suggested by the referee, and continuing to denote $r$ the minimal degree such that the homogeneous component $D_{0}(\mathcal{K})_{k}$ is non-zero, $\mathcal{K}$ is nearly free if and only if
	$$\tau(\mathcal{K}) = 62^{2} - r(62-r) - 1,$$
by \cite[Theorem 1.3]{DimcaTjurina}.
Since the equation again has no real roots, $\mathcal{K}$ cannot be nearly free.

\begin{remark}
\label{Dimca1}
Going further in the analysis, it is natural to ask how far $\mathcal{K}$ is from being free; to this end Dimca introduced the so called \textit{defect of freeness}, defined as 
	$$\nu (\mathcal{CL}) = {\rm max}_{k} \{ {\rm dim} N_{k}\}.$$
Checking whether a certain arrangement $\mathcal{CL}$ is free can be reduced to one of the following to conditions:
	\begin{itemize}
		\item ${\rm pdim}(S/J_{\mathcal{CL}}) = 2$, where $J_{\mathcal{CL}}$ denotes
		the Jacobian ideal and $S/J_{\mathcal{CL}}$ the Milnor algebra of 
		the defining polynomial of $\mathcal{CL}$;
		\item the Jacobian module $N := I_{\mathcal{CL}} / J_{\mathcal{CL}}$ vanishes, where $I_{\mathcal{CL}}$ denotes the the saturation of $J_{\mathcal{CL}}$ with respect to the irrelevant ideal $(x,y,z)$.
	\end{itemize}
	In the case of our Klein invariant conic-line arrangement $\mathcal{K}$, using a Singular script, we can compute the minimal resolution of the Milnor algebra $S/J_{\mathcal{K}}$, which has the following form:
	\begin{equation}\label{eq:nonfree}
	0 \rightarrow S(-121) \rightarrow S(-115) \oplus S(-99) \oplus S(-93)\rightarrow S^{3}(-62) \rightarrow S,
	\end{equation}
	Theorem 2.8 ii) in \cite{Dimca} describes the resolutions of Milnor algebras for nearly free curves; denoting by $(d_1,d_2,d_3)$ the degrees of the relations, a necessary condition is $d_1+d_2=\deg \mathcal{CL}$.
	
	This way we obtain an alternative computational check that $\mathcal{K}$ is neither free (because ${\rm pdim}(S/J_{\mathcal{K}}) = 3$) nor almost free (because $(d_{1},d_{2},d_{3}) = (31,37,53)$) thanks to \eqref{eq:nonfree}. 
	Additionally, we can compute the defect of freeness.
	The resolution of the Milnor algebra (\ref{eq:nonfree}) gives $r = 93 - 62 = 31$. Then by a recent result by Dimca \cite[Theorem 1.2 (2)]{DimcaV} we can calculate that
$$\nu(\mathcal{K}) = \bigg\lceil \frac{3}{4}(d-1)^{2} \bigg\rceil - \tau(\mathcal{K}) =  \bigg\lceil \frac{3}{4}\cdot62^{2} \bigg\rceil - 2751 = 132,$$
so $\mathcal{K}$ is far away from being free.
\end{remark}
\begin{remark}
\label{Dimca2}
If $E$ is the rank two vector bundle associated to the graded  $S$-module $D_{0}(\mathcal{K})$, then using a recent result by Abe and Dimca \cite[Theorem 1.1 (2)]{AbeDimca} we can show that the generic splitting type of $E$ (after we restrict $E$ to a generic line) is $(31,31)$.
\end{remark}
\section{Chern slopes}\label{sec:chern}
Let $X$ be a smooth projective surface and $D$ be a simple normal crossing divisor. A log-surface is a smooth surface $U :=X \setminus D$. We refer to $U$ as the pair $(X,D)$. The logarithmic Chern numbers of $(X,D)$ are defined as follows:
$$\overline{c}_{1}^{2}(X,D) = (K_{X}+D)^{2} \text{ and } \overline{c}_{2}(X,D) = e(X) - e(D).$$ 
We consider log-surfaces defined by plane curve arrangements, i.e., pairs $(\mathbb{P}^{2},\mathcal{F})$, where $\mathcal{F}$ is a arrangement of $k\geq 3$ smooth curves, each having degree $d \geq 1$, such that all singularities of $\mathcal{F}$ are ordinary and there is no point where all curves meet. Consider the blowing up $\pi : X \rightarrow \mathbb{P}^{2}$ along singular points of $\mathcal{F}$ having multiplicities $\geq 3$, and denote by $\mathcal{F}'$ the reduced total transform of $\mathcal{F}$ under $\pi$. Then $(X,\mathcal{F}')$ is a log-surface. We can compute the logarithmic Chern numbers of $(X,\mathcal{F}')$, namely:
$$\overline{c}_{1}^{2}(X,\mathcal{F}') = 9 + (d^{2}-6d)k + \sum_{r \geq 2}(3r-4) t_{r},$$
$$\overline{c}_{2}(X,\mathcal{F}') = 3 +(d^{2}-3d)k + \sum_{r\geq 2}(r-1)t_{r}.$$
In \cite[Theorem 3.1]{Pokora}, the second author proved the following result.
\begin{theorem}
	Let $\mathcal{F} \subset \mathbb{P}^{2}$ be a curve arrangement defined as above such that each irreducible component has degree $d \geq 2$, then for the associated log-surface $(X,\mathcal{F}')$ one has
	$$\overline{c}_{1}^{2}(X,\mathcal{F}') < \frac{8}{3} \overline{c}_{2}(X,\mathcal{F}').$$
\end{theorem}
However, it is not known whether the above bound is sharp. Unfortunately, we could not find in the literature interesting examples which could potentially lead to logarithmic Chern slopes $E(X,\mathcal{F}') :=\overline{c}_{1}^{2}(X,\mathcal{F}') / \overline{c}_{2}(X,\mathcal{F}')$ close to $8/3$. Focusing on the case of curve arrangements $\mathcal{F} \subset \mathbb{P}^{2}$ having degree $d=2$, Klein's arrangement $\mathcal{K}_2$ consists of $21$ conics with $224$ triple and $168$ double points as the intersections. Therefore
$$E(X, \mathcal{K}_2') \approx 2.25,$$
and this logarithmic Chern slope, to the best of our knowledge, is the largest known value in the class of curve arrangements of degree $2$.

\section{On the containment problem}\label{sec:containment}
In this section, we would like to briefly discuss the containment problem for symbolic and ordinary powers of homogeneous ideals of finite sets of points in the projective plane. Here we consider only a very special case of that problem, for a general introduction we refer to the survey \cite{SzSz}. Let $\mathcal{P} = \{p_{1}, ...,p_{s}\} \subset \mathbb{P}^{2}$ be a finite set of mutually distinct points and we define the following radical ideal 
$$I=I(\mathcal{P}) = I(p_{1}) \cap ... \cap I(p_{s}).$$
Then the $m$-th symbolic power of $I$ can be defined as
$$I^{(m)} = I^{m}(p_{1}) \cap ... \cap I^{m}(p_{s}).$$
Geometrically speaking, the $m$-th symbolic power is the set of homogeneous forms vanishing along $p_{i}$'s with multiplicity $\geq m$.
In recent years, and spurred by the work of Ein--Lazarsfeld--Smith \cite{ELS02}, Hochster and Huneke \cite{HH02} and Harbourne and Huneke \cite{HH}, the following problem has attracted considerable attention.
\begin{problem}
Determine the pairs $(m,r) \in \mathbb{Z}_{> 0}^{2}$ such that the containment 
$$I^{(m)} \subset I^{r}$$
holds.
\end{problem} 
\begin{theorem}[Ein--Lazarsfeld--Smith]\label{thm:els}
Let $I = I(p_{1}) \cap ... \cap I(p_{s})$ be a saturated ideal of a reduced finite set of points in $\mathbb{P}^{2}$, then the containment
$I^{(2r)} \subset I^{r}$ holds for all $r\ge 0$.
\end{theorem}

It is worth noting that the groundbreaking and elegant Theorem \ref{thm:els} holds regardless of the position of points. 
It is natural to ask whether the above result is sharp. In this direction, Huneke proposed the following problem.
\begin{question}[Huneke]\label{qu:huneke}
Let $I = I(p_{1}) \cap ... \cap I(p_{s})$ be a saturated ideal of a reduced finite set of points in $\mathbb{P}^{2}$. Does the containment
$I^{(3)} \subset I^{2}$
hold?
\end{question}
Most sets of points do satisfy Huneke's containment. However, in \cite{DSTG}, a first counterexample was found: the radical ideal of $12$ triple points of the dual-Hesse arrangement of $9$ lines. By \cite{BdRHHLPSz}, the radical ideal of the singular points of Klein's arrangement of lines $\mathcal{K}_1$ is also a counterexample to Huneke's question. 

S. Akesseh proved in \cite{Ake17} that the pullback of the ideal $I$ of a set of points in $\bbP^2$ by a morphism $\bbP^2 \rightarrow \bbP^2$, like our gradient map $\nabla(\Phi_{4})$, satisfies the same containment relations as $I$. 
By Propositions \ref{preimage_28} and \ref{pro:preimage21}, the pullback by $\nabla(\Phi_{4})$ of the radical ideal of the singular points of $\mathcal{K}_1$ is the radical ideal $I_3$ of the set of points of multiplicity at least 3 in the arrangement $\mathcal{K}$ of lines and conics, and this is a counterexample to Huneke's question. In fact, $\Phi_{63} \in I_3^{(3)} \setminus I_3^{2}$. 

It is more interesting to notice that the radical ideal $I_2$ of the whole set of singular points of $\mathcal{K}$ is a new counterexample to Question \ref{qu:huneke}. 
Indeed, it is clear that $\Phi_{6}\Phi_{63} \in I_2^{(3)}$, and we checked using Singular that $\Phi_{6}\Phi_{63}\notin I_2^2$.

In the next section we obtain an additional counterexample.
\section{Iterated preimages}\label{sec:iteration}

Several of the interesting properties we have shown for the arrangement $\mathcal{K}$ of reducible polars are consequences of its being the pullback of Klein's line arrangement $\mathcal{K}_1$. 
This suggests iterating the process, so we consider  
the preimage $\Phi_{189}=\nabla(\Phi_4)^*(\Phi_{63})$.
Since $\Phi_{63}=\nabla(\Phi_4)^*(\Phi_{21})=\Phi_{42}\Phi_{21}$, there is a splitting $\Phi_{189}=\Phi_{126}\Phi_{42}\Phi_{21}$, where $\Phi_{126}=\nabla(\Phi_4)^*(\Phi_{42})$ is an invariant arrangement of $21$ sextics. Further iterations lead to invariant arrangements of higher degree:
\[
(\nabla(\Phi_4)^k)^*(\Phi_{21})=\Phi_{14\cdot 3^{k}}\cdots 
\Phi_{42}\Phi_{21}.
\]
Our first observation is that each arrangement $V(\Phi_{14\cdot 3^{k}})$ contains the points of $\Orb_{42}$. This follows from the following lemma, and allows us to give a new type of counterexamples to Huneke's question \ref{qu:huneke}.

\begin{lemma}
$\nabla(\Phi_{4})(\Orb_{42})=\Orb_{42}$.
\end{lemma}
\begin{proof}
We saw in Proposition \ref{pro:preimage21} that $\nabla(\Phi_{4})(\Orb_{42})$ is disjoint from $\Orb_{21}$; since the image of $\Orb_{42}$ has to be an orbit of size dividing $42$, it follows that $\nabla(\Phi_{4})(\Orb_{42})=\Orb_{42}$.
\end{proof}

\begin{proposition}
	Let $T=\nabla(\Phi_{4})^{-2}(\Sing(\mathcal{K}_1))$.
	The ideal $I_X$ of every reduced set of points $X$ with 
	$$ T \subseteq X \subseteq T \cup \Orb_{42}.$$
	satisfies $\Phi_{189}\in I_X^{(3)}\setminus I_X^2$, so  $I_X^{(3)}\not\subseteq I_X^2$.
\end{proposition}
\begin{proof}
	Since $\Phi_{126}=(\nabla(\Phi_{4})^2)^*(\Phi_{21})$, it is clear that $\Phi_{126}$ has multiplicity 3 at every point of $T$. 
	By the previous lemma, $V(\Phi_{21})$, $V(\Phi_{42})$ and $V(\Phi_{126})$ pass through all points of $\Orb_{42}$, so $\Phi_{189}=\Phi_{126}\cdot\Phi_{42}\cdot\Phi_{21}$ vanishes at order 3 along $\Orb_{42}$, and $\Phi_{189}\in I_X^{(3)}$ for every $X\subseteq T \cup \Orb_{42}.$
	
	On the other hand, by the result of Akesseh \cite{Ake17} mentioned above, $\Phi_{189}\notin I_T^2$, and obviously $I_X^2\subseteq I_T^2$.
\end{proof}

It is worth pointing out that in this example, the same polynomial gives failure of containment for the whole set $T\cup \Orb_{42}$ of its points of multiplicity $\ge 3$, but also for sets strictly included in it, a phenomenon not previously described. Note however that nested sets of points whose ideals exhibit failure of containment was already observed in the case of Wiman's arrangement of lines  \cite[Example 3.3.5]{Har16}, where the containment $I^{(3)} \subset I^{2}$ still does not hold if we remove exactly one point of multiplicity $3$ from the set of all $201$ singular points. In the next section we present some details about Wiman's arrangement of lines and its further applications in the context of our work.

\medskip

These iterated pullbacks are interesting examples also for the question of bounded negativity. However, their singularities are not ordinary any longer. For instance, $V(\Phi_{63})$ is tangent to $V(\Phi_{42})$ at the points of $\Orb_{42}$. 

\begin{lemma}\label{lem:tangency}
	For every $k\ge 1$, $\Phi_{14\cdot 3^{k}}$ is smooth along $\Orb_{42}$.
	Moreover, $\Phi_{14\cdot 3^{k}}$ is tangent to $\mathcal{K}_2$ at every $p\in \Orb_{42}$.
\end{lemma}
\begin{proof}
	Choose a point $p \in \Orb_{42}$, and denote $p'=\nabla(\Phi_{4})(p)\in \Orb_{42}$.
	Let $L$, respectively $C$, be the component of $\mathcal{K}_1$, respectively $\mathcal{K}_2$, through $p$. 
	Choose an analytic neighborhood $U$ at $p$ and local coordinates $(x,y)$ in $U$ such that $L:\{x=0\}$, $C:\{y=0\}$. 
	Let $\alpha\in\Aut(\Phi_{4})$ be such that $\alpha(p)=p'$. 
	Then $\alpha(U)$ is an open neighborhood of $p'$, and $\tilde{x}=x \circ \alpha^{-1}$, $\tilde{y}=y \circ \alpha^{-1}$ are local coordinates, such that
	$\tilde x=0$, resp. $\tilde y=0$, is a local equation of the component of $\mathcal{K}_1$, resp. $\mathcal{K}_2$, through $p'$.
	There is no harm in identifying $U$ with $\alpha(U)$ and $(\tilde{x},\tilde{y})$ with $(x,y)$, and we shall do so in the sequel (in fact, there are points $p \in \Orb_{42}$ fixed by $\nabla(\Phi_{4})$, see the Singular computations in the appendix, and choosing such a $p$ even eliminates the abuse of notation.)
	
	Shrinking the neighborhood $U$ if necessary, we may assume that $\nabla(\Phi_{4})$ is given locally as $\nabla(\Phi_{4})(x,y)=(xy,g)$, where $g=g(x,y)\in \cO_{\bbP^2,p}$ is a power series in $x,y$ converging in $U$ to a local equation of $\nabla(\Phi_{4})^*(C)$, and hence of $\Phi_{126}$. 
	General polars of the Klein quartic passing through $p$ are smooth, hence the series $g$ is of order 1 (and $\Phi_{126}$ is smooth at $p$).
	On the other hand, since the discriminant $\Delta$ of $\nabla(\Phi_{4})$ is singular at $p'$ by Remark \ref{rmk:cusps_discriminant}, the local intersection multiplicity of $g(x,y)=0$ with the Jacobian of $\nabla(\Phi_{4})$,
	\[J\nabla(\Phi_{4})_p=\begin{vmatrix}
	y & x \\ \frac{\partial g}{\partial x}(x,y) & \frac{\partial g}{\partial x}(x,y)
	\end{vmatrix}\]
	is at least two. Hence the initial form of $g$ is either $ax$ or $ay$ for some $a\in \bbC$. 
	But $J\nabla(\Phi_{4})_p$ is a local equation at $p$ of the sextic $V(\Phi_{6})$, and we know that $V(\Phi_{6})$ is \emph{not} tangent to $\mathcal{K}_1$ along $\Orb_{42}$, so the initial form of $g$ is $ay$. 
	Therefore $\nabla(\Phi_{4})^*(C)$ is tangent to $C$ at $p$, i.e., $V(\Phi_{126})$ is tangent to $\mathcal{K}_2$. 
	Moreover, the preimage of every curve smooth at $p$ and tangent to $C$ has initial form $ay$ and hence will be smooth and tangent to $C$ as well, proving by recurrence that every $\Phi_{14\cdot 3^{k}}$ is tangent to $C$, hence to $\mathcal{K}_2$, at $p$.
\end{proof} 

\begin{remark}\label{rmk:cusps_K1}
	It follows from the computation of the Jacobian in the last paragraph of the proof that the singular polar $L+C:\{xy=0\}=\nabla(\Phi_{4})^*(L)$ has local intersection multiplicity 3 with the Jacobian, and therefore $L$ has intersection multiplicity 3 with the discriminant. In other words, $L$ is the tangent line to the cusp of $\Delta$ at $p$, as announced in \ref{rmk:cusps_discriminant}. 
\end{remark}

\begin{proposition}
	For every $k\ge 1$, $(\nabla(\Phi_{4})^k)^*(\mathcal{K}_1)$ is reduced.
\end{proposition}
\begin{proof}
	Assume $(\nabla(\Phi_{4})^k)^*(\mathcal{K}_1)$ is non-reduced.
	Then $(\nabla(\Phi_{4})^{k-1})^*(\mathcal{K}_1)$ must contain the (irreducible) discriminant curve $\Delta$. 
	But this has a cusp at each point of $\Orb_{42}$ (Remark \ref{rmk:cusps_discriminant}), whereas by the preceding lemma every component of $\mathcal{K}_{k-1}$ is smooth along $\Orb_{42}$, a contradiction.
\end{proof}

With a little more effort it is possible to compute the intersection multiplicities between branches of the $k$-th pullback of the Klein arrangement $\mathcal{K}_1$, and hence determine exactly the singularities at $\Orb_{21}$, $\Orb_{28}$, $\Orb_{42}$ and their preimages.
Thus we can give an upper bound $H_k$ for the Harbourne index of $(\nabla(\Phi_{4})^k)^*(\mathcal{K}_1)$, taking into account multiplicities at infinitely near points. It turns out that $H_k$ is a decreasing sequence converging to $-1283/410 \approx -3.1293$. 
We skip the cumbersome details.

\section{Remarks on Wiman's arrangement}\label{sec:wiman}

The largest possible primitive subgroup of $\PGL_3$, isomorphic to $A_6$, was found by H.~Valentiner in 1889 \cite{Val89}. 
Its ring of invariants was computed by A.~Wiman in 1896 \cite{wimanA}, see also \cite{wiman}; the invariant curve of the smallest degree is a unique sextic which we call \emph{Wiman's smooth sextic}\footnote{Not to be confused with the nodal sextic usually associated with Wiman, whose group of automorphisms is isomorphic to $S_5$.}, and it also comes with an extremely special arrangement of lines, $\mathcal{W}_1$, which consists of 45 lines with exactly 36 quintuple points, 45 quadruple points, and 120 triple points as singularities ---first described by F.~Gerbaldi in 1882 \cite{gerbaldi}. 
It has been studied in recent years with regard to Bounded Negativity and the containment problem, yielding for instance the most negative $H$-index known for a curve with ordinary singularities.
It turns out that Wiman's smooth sextic also exhibits a remarkable configuration $\mathcal{W}$ of 45 reducible polars, as we prove below.
Not surprisingly, the theory of polarity and the Steinerian curve of sextics is much less developed and understood than that for quartic curves. 
Thus, providing a complete theoretical description of the singularities of the reducible polars of Wiman's smooth sextic, as we did for Klein's quartic, is far beyond the scope of this work.
Instead, we now report on the relevant properties found computationally using Singular (scripts in the Appendix).

Wiman's smooth sextic can be defined by the equation
\[\Psi_6=10x^3y^3+9z(x^5+y^5)-45x^2y^2z^2-135z^4xy+27z^6\]

and its group of automorphisms is $\Aut(\Psi_6))\cong A_6$ \cite{wiman}. 
As for the case of Klein's curve, there is a group $G$ of $3\times 3$ matrices representing all automorphisms of $V(\Psi_6)$, with the particularity that in this case $G$ is a triple cover of $A_6$, as it contains the matrices $\omega^i I_3$, where $\omega$ is a third root of unity.
If coordinates are chosen adequately, also in this case the group of matrices $G$ is closed under transposition  (see \cite{SymmetricBU})  so we shall work with those coordinates in the appendix, although the equation $\overline{\Psi}$ of the sextic is then not as nice as $\Psi_6$, the one given by Wiman. 
 
The ring of invariants of $G$ is generated by $\Psi_6$, its Hessian $\Psi_{12}$, the bordered Hessian $\Psi_{30}$ of $\Psi_6$ and $\Psi_{12}$, and the Jacobian $\Psi_{45}$ of the three first invariants, which is the equation of the arrangement of 45 lines.

Also in this case the group $G$ is generated by involutions $\alpha$, each of which is a harmonic homology with center at one of the 45 quadruple points of the arrangement $\mathcal{W}_1$ and axis at one of the 45 lines.

\begin{lemma}
	If $\alpha\in \PGL_3(\bbC)$ is a general homology that fixes an irreducible curve $C\subset \bbP^2$ of degree $d>1$, and $C$ intersects the axis of the homology transversely, then the polar of $C$ with respect to the center of $\alpha$ is reducible, and contains the axis of $\alpha$.
\end{lemma}
\begin{proof}
	Let $p$ be the center and $L$ be the axis of $\alpha$. Since $\alpha$ fixes $C$ and $L$, it also fixes the tangent lines to the $d$ points of $C\cap L$. But all lines fixed by $\alpha$ other than $L$ pass through the center $p$. Therefore the $d$ points of $C\cap L$ belong to the polar curve $\partial_p(C)$, which has degree $d-1$. Hence $L$ is a component of $\partial_p(C)$.
\end{proof}

\begin{corollary}
	The polar of $\Psi_6$ with respect to each of the 45 quadruple points of $\mathcal{W}_1$ decomposes as $L+Q$ where $L$ is one of the 45 lines in $\mathcal{W}_1$ and $Q$ is a quartic.
\end{corollary}

A Singular computation shows that the quartic component of each of the 45 reducible polars is smooth, intersecting the linear component at 4 distinct points, disjoint from the set of singularities of $\mathcal{W}_1$. 
In contrast with the behavior of the reducible polars of Klein's quartic, here the different quartic curves are tangent at some points, namely each of them passes through 8 of the 72 inflection points, and the five quartics going through the same inflection point are tangent there (the gradient map $\nabla(\Psi_{6})$ is ramified at these points, which map to the quintuple points of $\mathcal{W}_1$).
Therefore the number of proper quintuple points on $\mathcal{W}$ is $5^2\cdot 36-72$, and there are $72$ additional infinitely near quintuple points.
Taking into account all multiple points, proper and infinitely near, the Harbourne index of $\mathcal{W}$ is smaller than that of $\mathcal{W}_1$, and in fact we have $h(\mathcal{W})=-1173/347\approx-3.38$.

We want to remark as well that each set of 4 nodes on one of the 45 reducible polars belongs to the (smooth) Hessian $V(\Psi_{12})$, and they have the same image under the Steiner map to the Steinerian, namely the point $p$ whose polar is singular on them. In particular, the Steinerian (which in this case is a curve of degree 48) has 45 points of multiplicity 4 at the quadruple points of $\mathcal{W}_1$.

\section*{Acknowledgments}
We want to thank I.~Dolgachev for sharing with us facts and references on the Steinerian curve, which were essential for completing this work. We would also like to thank an anonymous referee for crucial comments that allowed to improve Section 4 devoted to the freeness of arrangements, and for suggesting Remarks \ref{Dimca1} and \ref{Dimca2}.

This work was begun during the first author's visit at Universitat Aut\`{o}\-no\-ma de Barcelona, under the financial support of the Spanish MINECO grant MTM2016-75980-P, which also supports the second author.
During the project Piotr Pokora was supported by the programme for young researchers at Institute of Mathematics Polish Academy of Sciences. 
In the first phase of the project Piotr Pokora was a member of Institute of Algebraische Geometrie at Leibniz Universit\"at Hannover. 
\section*{Appendix}
Here we present our Singular script which allows to verify some of the above claims.

{\scriptsize
\begin{verbatim}
LIB "elim.lib";
LIB "poly.lib";

ring R=(0,e),(x,y,z),dp;        // The 21 points and lines are defined over
minpoly=e6+e5+e4+e3+e2+e+1;     // the cyclotomic field of 7th roots of 1
ideal m=x,y,z;


/////////          Definition of the invariant polynomials       ////////////////////

poly Phi4=x3y+y3z+z3x;          // Klein quartic
ideal jf=jacob(Phi4);           // partial derivatives
poly Phi6=-det(jacob(jf))/54;   // Hessian (sextic invariant)
ideal jfh=jf+ideal(Phi6);

matrix bh[4][4]=transpose(jacob(jfh)),jacob(Phi6);  // Bordered Hessian matrix
poly Phi14=det(bh)/9;           // Invariant of degree 14

poly Phi21=det(jacob(ideal(Phi4,Phi6,Phi14)))/14; // Configuration of 21 lines

/////////////////////////////////////////////////////////////////////////////////////

/////////        Definition of the small orbits                  ////////////////////
/////////        Explicit check of classically known statements  ////////////////////
/////////        of proposition 1.2                              ////////////////////

ideal Phi21S=jacob(jacob(Phi21));
Phi21S=radical(std(Phi21S));
hilb(std(Phi21S));              // Check that they are 49 = 21+28 pts

reduce(Phi4,std(Phi21S));       // None belongs to the Klein
reduce(Phi6,std(Phi21S));       // Nor the Hessian
reduce(Phi14,std(Phi21S));      // Nor the Bordered Hessian
ideal O24=Phi4,Phi6;
O24=std(O24);
hilb(std(radical(O24)));        // Check 24 distinct inflection points
reduce(Phi21,O24);              // They do not belong to the lines
reduce(Phi14,O24);              // They do not belong to the Bordered Hessian
ideal O56=Phi4,Phi14;
O56=std(O56);
hilb(std(radical(O56)));        // Check 56 distinct points
reduce(Phi21,O56);              // They do not belong to the lines
ideal O42=Phi6,Phi14;
O42=std(radical(std(O42)));
hilb(O42);                      // Check 42 tangency points
reduce(Phi21,O42);              // They do belong to the lines
ideal O21=jacob(jacob(jacob(Phi21)));
O21=std(radical(std(O21)));
hilb(O21);                      // Check 21 4-ple points
ideal O28=std(quotient(Phi21S,O21));
hilb(O28);                      // Check 28 3-ple points

/////////////////////////////////////////////////////////////////////////////////////

/////////       Definition of the gradient map & fixed points   /////////////////////

map f=R,jacob(Phi4);

poly Phi63=f(Phi21);            // Configuration of 21 cubics
poly Phi42=Phi63/Phi21;         // 21 conics
list ListQ=factorize(Phi42);    // Check that Phi42 splits as 21 conics
size(ListQ[1]);                 // (22 factors, the first is a number)

matrix fvars[2][3]=jf,x,y,z;
ideal fixf=minor(fvars,2);
fixf=std(fixf);
fixf=sat(fixf,m)[1];
hilb(fixf);                     // There are 13 fixed points (not G-invariant notion)
reduce(Phi21,fixf);             // They belong to the lines

ideal fixfonO21=std(fixf+O21);
hilb(fixfonO21);                // Three of the fixed points are among the 21
ideal fixfonPhi21S=std(fixf+Phi21S);
hilb(fixfonPhi21S);             // 7, so four of the fixed points are among the 28
ideal fixfon42=std(fixf+O42);
hilb(fixfon42);                 // 6 (all remaining fixed points). 
                                // It follows that the orbits O21, O28, O42 are fixed


/////////////////////////////////////////////////////////////////////////////////////

/////////        Steinerian          ////////////////////////////////////////////////

reduce(Phi4^3,O21);             // It is an invariant of degree 12, so it is a linear
reduce(Phi6^2,O21);             // combination of Phi4^3 and Phi6^2. 
poly Steiner=Phi6^2+4*Phi4^3;   // We know it has to go through O21 (with nodes)
                                // and this determines the coefficients.
ideal SSteiner=std(jacob(Steiner));
hilb(SSteiner);                 // 69 = 21 (nodes) + 2*24 (cusps)
SSteiner=std(radical(SSteiner));
hilb(SSteiner);                 // 45 indeed
reduce(SSteiner,O21); 
reduce(SSteiner,O24);           // Both contained, as expected.


/////////////////////////////////////////////////////////////////////////////////////

/////////        Singularities of the Configuration of polars   /////////////////////

ideal Phi42S=std(jacob(Phi42));
ideal Phi423=jacob(jacob(Phi42));
ideal Phi63S=std(jacob(Phi63));

ideal Phi63R=std(Phi63S+ideal(Phi6));
hilb(Phi63R);                    // The 42 nodes of the reducible polars
                                 // So, the only singularities are the nodes and
                                 // locally biholomorphic preimages of Sing(Phi21)

Phi423=std(Phi423);              // Check that the preimages are as claimed
hilb(Phi423);                    // 224 triple points for the conics
ideal Phi423L=std(Phi423+ideal(Phi21));
hilb (Phi423L);                  // 168 belong to the lines, so are quadruple for K
reduce(Phi423,O56);              // The remaining 56 are the bitangency points


/////////       Freeness of the configuration         ///////////////////////////////

resolution rs=res(Phi63S,0);     
intmat B=betti(rs);
print(B,"betti");                // As given in section on freeness


/////////       Failure of containment          /////////////////////////////////////

ideal I3=std(f(Phi21S));          // Points of multiplicity >=3 in K
I3=sat(I3,m)[1]; 
hilb(I3);                         // First generator in degree 24
ideal square=std(I3^2);
reduce(Phi63,square);             // Nonzero: Phi63 in symbolic cube, not in square

ideal I2=intersect(I3,O42);       // All singular points
I2=std(I2);
I2=sat(I2,m)[1];
square=std(I2^2)
reduce(Phi63*Phi6,square);        // Nonzero again

/////////////////////////////////////////////////////////////////////////////////////

/////////       Iteration        ////////////////////////////////////////////////////

ideal ff49=std(f(f49));          // Second iterate, preimage of 49 points 

poly Phi189=f(Phi63);
ideal Phi1893=intersect(ff49,O42);// Together with O42, triple points of Phi189
Phi1893=std(Phi189S);
Phi1893=sat(Phi189S,m)[1];
square=std(Phi1893^2);
reduce (Phi189,square);


/////////////////////////////////////////////////////////////////////////////////////
/////////////////////////////////////////////////////////////////////////////////////

/////////          Wiman arrangements       /////////////////////////////////////////

ring S=(0,h),(x,y,z),dp;          // The Wiman arrangement is defined over
minpoly=h8-h7+h5-h4+h3-h+1;       // the cyclotomic field of 15th roots of 1
ideal m=x,y,z;

number w=h5;                      // Primitive 3rd root of 1
factorize(x2-5);
number d=2h7-2h3+2h2-1;           // Square root of 5

poly Psi6=x6+y6+z6+(6*w*d+3*d+15)*x2y2z2+
  (-3*w*d-15*w-6*d)*(x4y2+y4z2+z4x2)/4+
  (-3*w*d+15*w+3*d+15)*(x2y4+y2z4+z2x4)/4;  // Wiman sextic
  
ideal jfw=jacob(Psi6);            // partial derivatives
poly Psi12=det(jacob(jfw))/675;   // Hessian (invariant of degree 12)
ideal jfhw=jfw+ideal(Psi12);

matrix bhw[4][4]=transpose(jacob(jfhw)),jacob(Psi12);  // Bordered Hessian matrix
poly Psi30=det(bhw)/6480;         // Invariant of degree 30

poly Psi45=-det(jacob(ideal(Psi6,Psi12,Psi30)))/216; // Configuration of 45 lines

map g=S,jacob(Psi6);              // Gradient map

list lines=factorize(Psi45);
poly L=lines[1][2];    // One of the linear components
poly RedPolar=g(L);               // Its preimage
factorize(RedPolar);              // contains the line
poly Q=RedPolar/L;
ideal QS=jacob(Q);
QS=std(QS);
sat(QS,m);                        // Q is smooth
 
reduce(Q,L);                      // Biquadratic with distinct roots
                                  // This is the nonzero discriminant:
(-12h7+12h5-24h4+12h3-12h2-24h+48)*(-12h7+12h5-24h4+12h3-12h2-24h+48)
-4*(6h7+6h5+3h4-6h3+6h2+3h+3)*(-3h7-9h5+3h4+3h3-3h2+3h);
                                  
ideal TP=jacob(jacob(Psi45));     // 201 Points on the configuration

// A Gröbner basis of TP or its preimage is impractical to compute, 
// but we can work with the set of points on a given line

reduce(Psi45,L);                  // 0 
poly L44=reduce(Psi45/L,L);
list factors=factorize(L44);      
ideal quadruple=L,factors[1][2];  // Choose one point on each singular orbit
ideal quintuple=L,factors[1][17];
ideal triple1=L,factors[1][4];
ideal triple2=L,factors[1][8];

ideal g4=g(quadruple),Psi12;
g4=std(g4);
hilb(g4);                         // No ramification over quadruples

ideal g5=g(quintuple),Psi12;
g5=std(g5);
hilb(g5);                         // There is ramification over quintuples

ideal g31=g(triple1),Psi12;
g31=std(g31);
hilb(g31);                        // No ramification 
ideal g32=g(triple2),Psi12;
g32=std(g32);
hilb(g32);                        // No ramification 

reduce(lines[1],quintuple);       // Identify lines through the 5-fold point
poly L2=lines[1][32];
poly L3=lines[1][33];
poly L4=lines[1][34];
poly L5=lines[1][35];

g5=radical(g5);
hilb(g5);                         // Two distinct ramif points over each 5-tuple
list r5=primdecGTZ(g5);
ideal i=r5[1][1];

reduce(g(L2),L2);
poly Q2=g(L2)/L2;

ideal intmult=i^4,Q,Q2;
intmult=std(intmult);
hilb(intmult);                    // Ordinary tangency (infinitely near points created)
\end{verbatim}}

{\footnotesize
	\bibliographystyle{plainurl}
\bibliography{LinesConics.bib}{}}

\bigskip
   Piotr Pokora,   
   Institute of Mathematics,
   Polish Academy of Sciences,
   ul. \'{S}niadeckich 8,
   PL-00-656 Warszawa, Poland. \\
\nopagebreak
\textit{E-mail address:} \texttt{piotrpkr@gmail.com, ppokora@impan.pl}
   
\bigskip
	Joaquim Ro\'{e},
	Universitat Aut\`{o}noma de Barcelona, Departament de Matem\`{a}tiques,
08193 Bellaterra (Barcelona), Spain. \\
\nopagebreak
\textit{E-mail address:} \texttt{jroe@mat.uab.cat}
\end{document}